\documentclass[a4paper,12pt]{amsart}
\usepackage[utf8]{inputenc}
\usepackage[T1]{fontenc}
\usepackage[english]{babel}
\usepackage{amssymb,amsthm,times}
\usepackage{xcolor}
\usepackage{hyperref}
\usepackage{cancel}
\usepackage{graphicx}
\usepackage{float}
\usepackage{epstopdf}

\newtheorem{Theo}{Theorem}[section]

\newtheorem{Lemm}[Theo]{Lemma}

\newtheorem{Conj}[Theo]{Conjecture}

\newcommand{\IDT}{\mathbb{T}^{\N}}

\newcommand{\N}{\mathbb{N}}

\title{\bf Optimal comparison of $\pmb{P}$-norms of Dirichlet polynomials}

\author{A. Defant}
\address{Institut f\"{u}r Mathematik. Universit\"{a}t Oldenburg. D-26111 Oldenburg (Germany).} \email{defant@mathematik.uni-oldenburg.de}

\author{A. P\'{e}rez}
\address{Departamento de Matem\'{a}ticas, Universidad de Murcia, Espinardo. 30100 Murcia (Spain).} \email{antonio.perez7@um.es}

\thanks{The research of the second author was partially done during a stay in Oldenburg (Germany) under the support of a PhD fellowship of ``La Caixa Foundation'', and of the projects of MINECO/FEDER (MTM2014-57838-C2-1-P) and Fundación S\'{e}neca - Regi\'{o}n de Murcia (19368/PI/14).
}



\begin{document}
\maketitle

\begin{abstract}
Let $1 \leq p < q <  \infty$. We show that
\[ \sup{\frac{\left\| D\right\|_{\mathcal{H}_{q}}}{\left\| D\right\|_{\mathcal{H}_{p}}}} = \exp{\left( \frac{\log{x}}{\log{\log{x}}} \left(\log{\sqrt{\frac{q}{p}}} + O\left(\frac{\log{\log{\log{x}}}}{\log{\log{x}}}\right)\right) \right)} \,,\]
where the supremum is taken over all non-zero Dirichlet polynomials of the form $D(s)=\sum_{n \leq x}{a_{n} n^{-s}}$. An aplication is given to the study of multipliers between Hardy spaces of Dirichlet series.
\end{abstract}

\section{Introduction}

Let $1 \leq p < \infty$. Given a Dirichlet polynomial $D(s) = \sum_{n}{a_{n}n^{-s}}$, its $p$-norm is defined as
\begin{equation}
\label{equa:pNorm}
\|D\|_{\mathcal{H}_{p}} := \lim_{T \rightarrow \infty}{\left(\frac{1}{2T}\int_{-T}^{T}{|D(it)|^{p} \: dt}\right)^{1/p}}.
\end{equation}
The fact that the previous limit exists, can be argued by means of Bohr's one-to-one correspondence between Dirichlet series and (formal) power series in infinitely many variables \cite{BohrStrip}. Using that every $n \in \N$ has a unique prime number decomposition  $n=\mathfrak{p}^{\alpha} := \mathfrak{p}_{1}^{\alpha_{1}} \mathfrak{p}_{2}^{\alpha_{2}}\ldots$, where $\mathfrak{p} = (\mathfrak{p}_{n})_{n \in \N}$ is the ordered sequence of primes and $\alpha \in \N_{0}^{(\N)}$, the set of eventually null sequences in $\N_{0}:=\N \cup \{ 0\}$.
Following Bohr \cite{BohrStrip} we can identify every Dirichlet series $D = \sum_{n}{a_{n}n^{-s}}$ with the (formal) power series
\[ \mathcal{L}D \equiv \sum_{\alpha \in \N_{0}^{(\N)}}{a_{\mathfrak{p}^{\alpha}} z^{\alpha}}\,, \]
the so-called Bohr lift.
In case $D(s)$ is a Dirichlet polynomial, $\mathcal{L}D$ is then a trigonometric polynomial. And if $d \omega$ denotes the Haar measure on the infinite-dimensional torus $\mathbb{T}^{\N}$, Birkhoff Ergodic Theorem implies that the limit in \eqref{equa:pNorm} exists, being $\| D\|_{\mathcal{H}_{p}} = \| \mathcal{L} D\|_{L_{p}(\mathbb{T}^{\N})}$ (see \cite{Bayart1} for the details). This shows in particular that $\| \cdot \|_{\mathcal{H}_{p}}$ is a norm on the space of Dirichlet polynomials; and moreover, its completion $\mathcal{H}_{p}$ can be seen as a Banach space of Dirichlet series isometric to the Hardy space $H_{p}(\IDT)$ (defined as in \cite{ColeGamelin}) through Bohr's identification.  The systematic study of the Banach spaces $\mathcal{H}_p$
started in \cite{Bayart1} and  \cite{HardySpaceDirich}. Recall that in the setting of almost periodic functions, this type of limit was firstly considered by Besicovitch \cite{Besicovitch1}.

Given $1 \leq p,q < \infty$ we define
\[ \mho(q,p,x) := \sup{\frac{\| D\|_{\mathcal{H}_q}}{\| D\|_{\mathcal{H}_p}}} \hspace{3mm} \mbox{ taken over all $0 \neq D(s) = \sum_{n \leq x}{\frac{a_{n}}{n^{s}}}$}. \]
Along the paper we will always assume that $1 \leq p < q < \infty$, since this is the interesting case.

Let us introduce some notation: given $x > 0$ big enough, we define recursively $\log_{1}{x}:=\log{x}$ and $\log_{k}{x}:=\log_{k-1}{\log{x}}$ for $k>1$. The main result of the paper reads as follows:
\begin{Theo}
\label{Theo:mainEquality} For every $1 \leq p < q < \infty$
\[ \mho(q,p,x) = \exp{\left( \frac{\log{x}}{\log_{2}{x}} \left(\log{\sqrt{\frac{q}{p}}} + O\left(\frac{\log_{3}{x}}{\log_{2}{x}}\right)\right)  \right)}. \]
\end{Theo}

There already exist inequalities comparing the $p$-norms of certain type of trigonometric polynomials on $\IDT$
\[ P(z) = \sum_{\alpha \in \N_{0}^{(\N)}}{c_{\alpha} z^{\alpha}}. \]
Recall that $P(z)$ is said to be $m$-homogeneous (for some $m \in \N$) if $c_{\alpha} = 0$ whenever $|\alpha| := \alpha_{1} + \alpha_{2} + \ldots \neq m$. Let us denote
\[
H^{q,p}_{m} := \sup{\frac{\| P\|_{L_q(\IDT)}}{\| P\|_{L_{p}(\IDT)}}} \,,
\]
where the supremum is taken over all $m$-homogeneous polynomials $P(z) \neq 0$. Basing on Weissler' result \cite{Weissler} about hypercontractive estimates for the Poisson semigroup, Bayart \cite[Theorem 9]{Bayart1} proved that
\begin{equation}
\label{equa:estimationExponent}
H^{q,p}_{m} \leq \left(\sqrt{\frac{q}{p}}\right)^{m}.
\end{equation}
Recently, it has been shown that the best constant $C > 0$ such that $H^{q,p}_{m} \leq C^{m}$, is precisely $C = \sqrt{q/p}$ (see \cite{DefantMastilo}). We can deduce from the previous estimation that every polynomial $P(z)$ as above satisfies
\begin{equation}
\label{equa:pNormPolynomials1}
 \left\| P \right\|_{L_{q}(\IDT)} \leq H^{q,p}_{\deg{(P)}} \left\| P \right\|_{L_{p}(\IDT)}\,, \mbox{ where } \deg{(P)} := \max{\{ |\alpha| \colon c_{\alpha} \neq 0 \}}.
\end{equation}
Indeed, the rotation invariance of the Haar measure yields that
\begin{equation}
\label{equa:trickhomogeneization}
\| P \|_{L_{r}(\IDT)} = \| \tilde{P} \|_{L_{r}(\mathbb{T} \times \IDT)}\, \mbox{ for each $1 \leq r < \infty$}
\end{equation}
where $\tilde{P}$ is the trigonometric polynomial on $\mathbb{T} \times \IDT \equiv \IDT$ given by
\[ \tilde{P}(z,\omega) = z^{\deg(P)}P(\omega_{1}z^{-1}, \omega_{2}z^{-1}, \ldots)\,, \hspace{3mm} (z,\omega) \in \mathbb{T} \times \IDT. \]
But $\tilde{P}$ is an $m$-homogeneous polynomial with $m = \deg(P)$, so can apply \eqref{equa:estimationExponent} to $\tilde{P}$ and use \eqref{equa:trickhomogeneization} to conclude that \eqref{equa:pNormPolynomials1} holds. Using Bohr's lift, we can reformulate this last inequality in terms of Dirichlet polynomials $D(s)=\sum_{n}{a_{n}n^{-s}}$ as
\begin{equation}
\label{equa:pNormDirichletPolynomials1}
 \left\| D \right\|_{\mathcal{H}_{q}} \leq H^{q,p}_{m} \left\| D \right\|_{\mathcal{H}_{p}}\,, \mbox{ where } m = \max{\{ \Omega(n) \colon a_{n} \neq 0 \}}.
\end{equation}
Recall that $\Omega(n) = \Omega(\mathfrak{p}^{\alpha}) = |\alpha|$ is the function which counts the number of prime divisors of $n$ (with multiplicity). It satisfies $\Omega(n) \leq \log{n}/\log{2}$, which let us deduce that
\[ \mho(q,p,x) \leq \exp{\left( \frac{\log{x}}{\log{2}} \log{\sqrt{\frac{q}{p}}} \right)}. \]
Nevertheless, this upper bound is far from being optimal: A well-known inequality due to Helson \cite{Helson} together with an old estimation of $\max{\{ d(n) \colon n \leq x \}}$ in terms of $x$ due to Wigert \cite{Wigert}, gives that
\[
\begin{split}
\left\| \sum_{n \leq x}{\frac{a_{n}}{n^{s}}} \right\|_{\mathcal{H}_{2}} & \leq \left(\sum_{n \leq x}{\frac{|a_{n}|^{2}}{d(n)}}\right)^{1/2} \sqrt{ \max{\{ d(n) \colon n \leq x \}}}\\
& \leq \left\| \sum_{n \leq x}{\frac{a_{n}}{n^{s}}} \right\|_{\mathcal{H}_{1}} \exp{ \left( \frac{\log{x}}{\log_{2}{x}} \left(\log{\sqrt{2}} + O\left(\frac{\log_{3}{x}}{\log_{2}{x}}\right)\right) \right)}.
\end{split}
\]
This is the upper estimate for the special case $\mho(2,1,x)$  given in Theorem \ref{Theo:mainEquality}, and hence in this case it remains to prove the lower estimate. But in the general case the estimate for $\mho(q,p,x)$ needs a more delicate
argument which is carried out in Section \ref{sec:above}. It relies on a decomposition method inspired by  \cite{KonyaginQueffelec}, in combination with  \eqref{equa:estimationExponent} and a deep
number theoretical result of Bruijn. Section \ref{sec:below} deals with the construction of a suitable family of Dirichlet polynomials to obtain the lower estimate for $\mho(q,p,x)$. We follow an argument based on the Central Limit Theorem, which was used in  \cite{KwapienKonig} to give optimal bounds for the  constants in the Khintchine-Steinhaus inequality. To adapt this idea to our problem, we have to develop a quantitative result concerning the convergence of the $p$-moments for the special sequence of random variables we handle (Lemma \ref{Lemm:boundBelow} and Theorem \ref{Theo:comparingNorms}).

\section{Estimation from above} \label{sec:above}

\noindent Here we prove the upper estimate from Theorem \ref{Theo:mainEquality}:
\begin{equation}
\label{Theo:estimationAbove}
\mho(q,p,x) \leq \exp{\left( \frac{\log{x}}{\log_{2}{x}} \left(\log{\sqrt{\frac{q}{p}}} + O\left(\frac{\log_{3}{x}}{\log_{2}{x}}\right)\right) \right)}
\end{equation}

\begin{proof}
Fix $2 \leq y \leq x$ and denote
\begin{align*}
& S(x,y) = \{ n \leq x \colon \mathfrak{p}_{t}|n \Rightarrow \mathfrak{p}_{t} \leq y \}.
\\
& L(x,y) = \{ n \leq x \colon \mathfrak{p}_{t}|n \Rightarrow \mathfrak{p}_{t} > y \}.
\end{align*}
Let $D(s) = \sum_{n \leq x}{a_{n} n^{-s}}$ be a Dirichlet polynomial. Since each $1 \leq n \leq x$ can be uniquely decomposed as a product $n = j k$ for some $j \in S(x,y)$ and $k \in L(x,y)$, we can write
\begin{equation}
\label{equa:decompostionPrimeDivisors}
D(s) =  \sum_{j \in S(x,y)}{D_{j}(s) j^{-s}}\, \mbox{ where } D_{j}(s) = \sum_{k \in L(x,y)}{a_{j k} k^{-s}}.
\end{equation}
We claim that $\| D_{j}\|_{\mathcal{H}_{p}} \leq \| D\|_{\mathcal{H}_{p}}$ for every $p \geq 1$. To prove it, we will use Bohr's lift and translate the previous elements into trigonometric polynomials. Let $P = \mathcal{L}D$ be the  trigonometric polynomial
\[ P(\omega) = \sum_{\alpha \in \N_{0}^{N}}{a_{\mathfrak{p}^\alpha} \omega^{\alpha}}.\]
 If $\lambda := \pi(y) $, each $\alpha \in \N_{0}^{(\N)}$ has the form $\alpha = (\gamma,\beta)$ where $\gamma \in \N_{0}^{\lambda}$ , $\beta \in \N_{0}^{(\N)}$. Hence, for $\omega = (u,v) \in \mathbb{T}^{\lambda} \times \mathbb{T}^{\N} = \mathbb{T}^{\N}$
\[ P(u,v) = \sum_{\gamma \in \N_{0}^{\lambda}}{P_{\gamma}(v) u^{\gamma}} \hspace{2mm} \mbox{ where } \hspace{2mm} P_{\gamma}(v) = \sum_{\beta \in \N_{0}^{(\N)}}{c_{(\gamma,\beta)} v^{\beta}}.\]
For each $j \in S(x,y)$ we have that $\mathcal{L}D_{j} = P_{\gamma}$ whether $j = \mathfrak{p}^{(\gamma, 0)}$. Hence
\begin{align*}
\begin{split}
\| D_{j}\|^{p}_{\mathcal{H}_{p}} =  \|P_\gamma\|_{L_p(\mathbb{T}^{N-\lambda})}^p
&
= \int_{\mathbb{T}^{N - \lambda}}{\left|\int_{\mathbb{T}^{\lambda}}{P(u,v) u^{-\gamma} du}\right|^{p} dv}
\\&
 \leq \int_{\mathbb{T}^{N-\lambda}}  \int_{\mathbb{T}^{\lambda}}  |P(u,v)|^p du dv  = \|P\|_{L_p(\mathbb{T}^{N})}^p = \| D\|_{\mathcal{H}_{p}}^{p}  \,.
 \end{split}
\end{align*}
This proves the claim. Notice that every $k \in L(x,y)$ satisfies $x \geq k \geq y^{\Omega(k)}$. Combining this inequality with \eqref{equa:pNormDirichletPolynomials1} and \eqref{equa:estimationExponent}, for each $j \in S(x,y)$ we have that
\[ \left\|  D_{j}\right\|_{\mathcal{H}_q} \leq \exp{\left( \frac{\log{x}}{\log{y}}\log{ \sqrt{\frac{q}{p}} }\right)} \, \left\|  D \right\|_{\mathcal{H}_p}.   \]
Applying this to \eqref{equa:decompostionPrimeDivisors}, we get
\begin{equation}
\label{equa:optimizable}
\left\|  D \right\|_{\mathcal{H}_q} \leq \sum_{k \in S(x,y)}{\left\| D_{j} \right\|_{\mathcal{H}_q}} \leq |S(x,y)| \, \exp{\left( \frac{\log{x}}{\log{y}}\log{ \sqrt{\frac{q}{p}} }\right)} \, \| D\|_{\mathcal{H}_p}.
\end{equation}
A deep result due to Bruijn \cite[p. 359, Theorem 2]{Tenenbaum} states that
\begin{equation}
\label{equa:Bruijn}
\log{|S(x,y)|} =  Z\left(1 + O\left(\frac{1}{\log{y}} + \frac{1}{\log_{2}{x}}\right)  \right)
\end{equation}
uniformly for $2 \leq y \leq x$, where
\[ Z =Z(x,y):=\frac{\log{x}}{\log{y}} \log{\left( 1 + \frac{y}{\log{x}} \right)} + \frac{y}{\log{y}} \log{\left( 1 + \frac{\log{x}}{y} \right)}. \]
We choose a proper value of $y$ to minimize the constant in \eqref{equa:optimizable},
\[ y:= \exp{\left( \frac{(\log_{2}{x})^{2}}{\log_{2}{x} + \log_{3}{x}}\right)} = \frac{\log{x}}{\log_{2}{x}} \exp{\left(\frac{(\log_{3}{x})^{2}}{\log_{2}{x} + \log_{3}{x}}\right)}\,, \]
Notice that
\begin{align*}
\frac{y}{\log{y}} = \frac{\log{x}}{\log_{2}{x}} O\left(\frac{1}{\log_{2}{x}}\right) \: \mbox{ and } \: \log{\left( 1 + \frac{\log{x}}{y} \right)} = O(\log_{3}{x}).
\end{align*}
Using that $\log{(1 + t)} \leq t$ for each $t > 0$, we can bound
\[ Z \leq \frac{y}{\log{y}} \left( 1 + \log{\left( 1 + \frac{\log{x}}{y} \right)} \right) \: \mbox{ and so } \: Z = \frac{\log{x}}{\log_{2}{x}} O\left(\frac{\log_{3}{x}}{\log_{2}{x}}\right). \]
Using this estimation in \eqref{equa:Bruijn}, we get that
\begin{equation}
\label{equa:inequalityAux1}
\log{|S(x,y)|} = \frac{\log{x}}{\log_{2}{x}} O\left(\frac{\log_{3}{x}}{\log_{2}{x}}\right).
\end{equation}
On the other hand, for the taken value of $y$
\begin{equation}
\label{equa:inequalityAux2}
\exp{\left( \frac{\log{x}}{\log{y}} \log{\sqrt{\frac{q}{p}}} \right)} = \exp{\left( \frac{\log{x}}{\log_{2}{x}} \left( \log{\sqrt{\frac{q}{p}}} + \frac{\log_{3}{x}}{\log_{2}{x}} \right) \right)}.
\end{equation}
Replacing estimations \eqref{equa:inequalityAux1} and \eqref{equa:inequalityAux2} in \eqref{equa:optimizable}, we conclude the result.
\end{proof}

\section{Estimation from below}
\label{sec:below}

\noindent Along this section we will denote for every $n \in \N$ and $z \in \mathbb{C}^n$
\[ Q_{n}(z) = \frac{1}{\sqrt{n}} \sum_{j=1}^{n}{z_{j}}. \]
A special case of the Khinchine-Steinhaus inequality given in \cite[Theorem 2]{KwapienKonig} states that for every
 $r \geq 1$ and every $n$
\[
\int_{\mathbb{T}^{n}}{\left|Q_{n}(z) \right|^{2r} \: dz}
\leq \Gamma(r + 1)\,.
\]
Let us point out that here the constant on the right side of this inequality is independent of $n$, and even optimal since by the central limit theorem
\begin{equation*}
\lim_{n \rightarrow \infty }\int_{\mathbb{T}^{n}}{\left|Q_{n}(z) \right|^{2r} \: dz} =  \Gamma(r + 1).
\end{equation*}
Hence by Stirling's formula for every  $ r \geq 1$
\begin{equation}
\label{equa:optimalKhintchine}
\int_{\mathbb{T}^{n}}{\left|Q_{n}(z) \right|^{2r} \: dz} \leq  \sqrt{2 \pi r} \left(\frac{r}{e}\right)^{r} e^{\frac{1}{12r}}\,.
\end{equation}
 We will need a similar lower estimate.
\begin{Lemm}
\label{Lemm:boundBelow}
For $m,n \in \mathbb{N}$ with every $n > m+1$ we have
\[ \int_{\mathbb{T}^{n}}{\left| Q_{n}(z)\right|^{2m} \: dz} \geq \sqrt{2 \pi m} \left(\frac{m}{e}\right)^{m} e^{\frac{-4m^{2}}{n}} . \]
\end{Lemm}

\begin{proof}
Using the multinomial formula we have that
\[  Q_{n}(z) ^{m} =  \sum_{\alpha \in \N_{0}^{n}, |\alpha| =m}{\frac{m!}{\alpha!} \frac{z^{\alpha}}{\sqrt{n}^{m}}}\,,\]
(as usual we here  write $\alpha!= \prod_j \alpha_{j}!$), hence by integration (and the orthogonality of the monomials on $\mathbb{T}^n$)
\begin{equation*}
\label{equa:integralEquality}
 \int_{\mathbb{T}^{n}}{\left| Q_{n}(z) \right|^{2m} \: dz} = \frac{1}{n^{m}} \sum_{\alpha \in \N_{0}^{n},|\alpha|=m}{\left(\frac{m!}{\alpha!} \right)^{2}}\,.
\end{equation*}
Now we make use of the Cauchy-Schwartz inequality and again the multinomial formula to deduce that
\begin{align*}
\Big(\sum_{\alpha \in \N_{0}^{n},|\alpha|=m}{\Big(\frac{m!}{\alpha!}
 \Big)^{2}}\Big)^{1/2}
\Big(\sum_{\alpha \in \N_{0}^{n}, |\alpha|=m}{1}\Big)^{1/2} \geq \sum_{\alpha \in \N_{0}^{n},|\alpha|=m}{\frac{m!}{\alpha!}}
 = n^{m}\,,
\end{align*}
and since
\[
 \sum_{\alpha \in \N_{0}^{n}, |\alpha|=m}1\,
=\, \binom{m+n-1}{m}\,,
\]
we arrive at
\begin{equation} \label{equa:squaresMultinomialCoeff}
 \int_{\mathbb{T}^{n}}{\left| Q_{n}(z) \right|^{2m} \: dz} \geq n^m \binom{m+n-1}{m}^{-1}\,.
\end{equation}
In order to  be able to  handle the binomial coefficient we need the followig  estimate
\begin{equation}\label{Lemm:EstimationCombinatorialNumber}
\binom{m+n-1}{m} \leq \frac{1}{\sqrt{2\pi}} \, \sqrt{\frac{m +n-1}{(n-1) \, m}} \, \frac{( m+n-1)^{m+n-1}}{(n-1)^{n-1} \, m^m}\,;
\end{equation}
indeed, by Stirling's formula for every $k$
\[ \sqrt{2 \pi k} \, \left( \frac{k}{e} \right)^{k} e^{\frac{1}{12k+1}} < k! < \sqrt{2 \pi k} \, \left( \frac{k}{e} \right)^{k} e^{\frac{1}{12k}} \]
which gives
\begin{equation*}
\begin{split}
\frac{(m+n-1)!}{(n-1)! \: m!} & \leq \frac{\sqrt{2 \pi (m+n-1)} \, \left( \frac{m+n-1}{e} \right)^{m+n-1} \, e^{\frac{1}{12(m+n-1)}} }{\sqrt{2 \pi (n-1)} \, \left( \frac{n-1}{e} \right)^{n-1} \, e^{\frac{1}{12(n-1)+1}} \, \sqrt{2\pi m} \, \left( \frac{m}{e} \right)^{m} \, e ^{\frac{1}{12m+1}} }\,,
\end{split}
\end{equation*}
and consequently \eqref{Lemm:EstimationCombinatorialNumber}. We combine now  \eqref{equa:squaresMultinomialCoeff} and
\eqref{Lemm:EstimationCombinatorialNumber} to obtain
\[
\begin{split}
\int_{\mathbb{T}^{n}}{\left| Q_{n}(z) \right|^{2m} \: dz} & \geq \sqrt{2 \pi} \,n^{m} \sqrt{\frac{m(n-1)}{m+n-1}} \frac{   (n-1)^{n-1}m^{m}}{(m+n-1)^{m+n-1}}\\
& = \sqrt{2 \pi m} \frac{m^{m}}{\left( 1 +\frac{m}{n-1} \right)^{n-1}} \sqrt{\frac{n-1}{m+n-1}} \left(\frac{n}{m+n-1}\right)^{m}  \\
& \geq \sqrt{2 \pi m} \left(\frac{m}{e}\right)^{m} \left(\frac{n-1}{m+n-1}\right)^{m + 1/2}\,;
\end{split}
\]
for the last esimate we  use that $(1+ 1/x)^{x} \leq e$ for $x \geq 1$. To bound the last factor, we use that $(1 - 1/x)^{x} > e^{-2}$ for $x > 2$, so that
\[ \left(\frac{n-1}{m+n-1}\right)^{m + 1/2}  \geq e^{-2 \frac{m(m+1/2)}{m+n-1}} \geq e^{\frac{-4m^{2}}{n}}. \]
This completes the argument.\end{proof}

To simplify the notation, from now on given two functions $f,g$ depending on $p,q$ and probably other variables, we will write $f \gg g$ when $f \geq c \, g$ for some constant $c=c(p,q)$ depending on $p$ and $q$ but independent of the rest of variables.

\begin{Theo}
\label{Theo:comparingNorms}
Let $n,k \in \mathbb{N}$ with $n > [kq/2] + 1 > [kp/2] + 1 > 1$. Then
\[ \frac{\|Q_{n}^{k}\|_{q}}{\|Q_{n}^{k}\|_{p}} \gg k^{\frac{1}{2q} - \frac{1}{2p}} \left(\frac{q}{p}\right)^{k/2} e^{\frac{-q k^{2}}{n}}. \]
\end{Theo}

\begin{proof}
Since $kp \geq 2$ by hypothesis, we can use \eqref{equa:optimalKhintchine} bound
\begin{equation}
\label{equa:boundpnorm}
\|Q_{n}^{k}\|_{p} \leq (\pi k p)^{\frac{1}{2p}} \left( \frac{kp}{2e} \right)^{\frac{k}{2}} e^{\frac{1}{6 k p^{2}}}.
\end{equation}
On the other hand, we want to give a lower bound of $\|Q_{n}^{k}\|_{q}$. Let $m:=[k q/2] \geq 1$. Since $kq \geq 2m$, we can write
\begin{equation}
\label{equa:auxBoundQuotiento}
\|Q_{n}^{k}\|_{q}^{q} = \int_{\IDT}{|Q_{n}(\omega)|^{kq} \: d \omega} \geq \left( \int_{\IDT}{\left| Q_{n}(\omega) \right|^{2m} \: d\omega} \right)^{\frac{k q}{2m}}.
\end{equation}
We can then use the lower bound of Lemma \ref{Lemm:boundBelow} in \eqref{equa:auxBoundQuotiento} to deduce that
\begin{equation}
\label{equa:boundqnorm}
\| Q_{n}^{k}\|_{q} \geq (\sqrt{2 \pi m})^{k/2m} \left(\frac{m}{e}\right)^{k/2} e^{-\frac{2k m}{n}}.
\end{equation}
Combining \eqref{equa:boundqnorm} and \eqref{equa:boundpnorm} we arrive to
\begin{equation}
\label{equa:auxBoundQuotient1}
\frac{\|Q_{n}^{k}\|_{q}}{\|Q_{n}^{k}\|_{p}} \geq (\sqrt{2\pi})^{\frac{k}{2m} - \frac{1}{p}} \cdot \frac{m^{\frac{k}{4m}}}{\left( \frac{kp}{2} \right)^{\frac{1}{2p}}} \cdot \frac{m^{\frac{k}{2}}}{\left(\frac{kp}{2}\right)^{\frac{k}{2}}} \cdot e^{\frac{-1}{6 k p^{2}}} e^{-\frac{2k m}{n}} .
\end{equation}
Using again $k q/2 \geq m \geq kq/2-1$, we get that
\begin{equation}
\label{equa:auxBoundQuotient2}
\frac{m^{\frac{k}{4m}}}{\left( \frac{kp}{2} \right)^{\frac{1}{2p}}} \gg k^{\frac{1}{2q} - \frac{1}{2p}} \cdot \frac{m^{\frac{k}{4m}}}{\left( \frac{k q}{2} \right)^{\frac{1}{2q}}} \geq  k^{\frac{1}{2q} - \frac{1}{2p}} \cdot \frac{m^{\frac{1}{2q}}}{\left( \frac{k q}{2} \right)^{\frac{1}{2q}}} \gg k^{\frac{1}{2q} - \frac{1}{2p}}
\end{equation}
\begin{equation}
\label{equa:auxBoundQuotient3}
\frac{m^{\frac{k}{2}}}{\left(\frac{kp}{2}\right)^{\frac{k}{2}}} \gg \left(\frac{q}{p}\right)^{\frac{k}{2}} \frac{m^{\frac{k}{2}}}{(\frac{kq}{2})^{\frac{k}{2}}} \gg \left(\frac{q}{p}\right)^{\frac{k}{2}}.
\end{equation}
Applying \eqref{equa:auxBoundQuotient2} and \eqref{equa:auxBoundQuotient3} to \eqref{equa:auxBoundQuotient1} we can conclude that
\[ \frac{\|Q_{n}^{k}\|_{q}}{\|Q_{n}^{k}\|_{p}} \gg k^{\frac{1}{2q} - \frac{1}{2p}}  \left(\frac{q}{p}\right)^{\frac{k}{2}}  e^{-\frac{2k m}{n}} \geq  k^{\frac{1}{2q} - \frac{1}{2p}}  \left(\frac{q}{p}\right)^{\frac{k}{2}}  e^{\frac{-q k^{2}}{n}}\,, \]
which is what we wanted.
\end{proof}

The trigonometric polynomial $Q_{n}^{k} = \sum_{\alpha}{c_{\alpha}z^{\alpha}}$ satisfies that $c_{\alpha} \neq 0$ if and only if $\alpha \in \N_{0}^{n}$ with $|\alpha| = k$. Let us fix a real number $x > e^{e^{e}}$ and consider the values
\[ k(x):=\left[\frac{\log{x}}{\log_{2}{x} + \log_{3}{x}}\right] \: \mbox{ and } \: n(x):=\pi(x^{1/k(x)}).\]
The correspondent Dirichlet series via Bohr transform is then of the form
\[ D_{x}(s) = \mathcal{L}^{-1}Q_{n(x)}^{k(x)}= \left(\sum_{i=1}^{n(x)}{\frac{1}{\sqrt{n} \mathfrak{p}_{i}^{s}}}\right)^{k(x)} = \sum_{m \leq x}{a_{m} m^{-s}}. \]

\begin{Theo} For each $1 \leq p < q < \infty$
\[ \frac{\|D_{x}\|_{q}}{\|D_{x}\|_{p}} \geq \exp{\left( \frac{\log{x}}{\log_{2}{x}} \left(\log{\sqrt{\frac{q}{p}}} + O\left(\frac{\log_{3}{x}}{\log_{2}{x}}\right)\right)  \right)}. \]
\end{Theo}

\begin{proof}
Using the prime number theorem, and more especifically a bound  due to Dusart \cite[Theorem 1.10]{Dursat}, we have that for $x^{1/k(x)} \ge 599$
\[ \pi(x^{1/k(x)}) \geq \frac{k(x) x^{1/k(x)}}{\log{x}} \left( 1 + \frac{k(x)}{\log{x}}\right). \]
Therefore
\[ \frac{k(x)}{n(x)} \leq \frac{\log{x}}{x^{1/k(x)}} = \exp{\left( \log_{2}{x} - \frac{\log{x}}{k(x)} \right)} \leq \exp{\left( - \log_{3}{x} \right)} = \frac{1}{\log_{2}{x}}. \]
Note that $n(x)/k(x)$ tends to infinity when $x$ does, so for $x$ big enough the hypothesis of Theorem \ref{Theo:comparingNorms} are satisfied. This means that we can bound
\[ \frac{\|D_{x}\|_{q}}{\|D_{x}\|_{p}} \gg k(x)^{\frac{1}{2q} - \frac{1}{2p}} \left(\frac{q}{p}\right)^{k(x)/2} e^{\frac{-k(x)^{2}q}{n(x)}} = \exp{\left( k(x)  \left(\log{\sqrt{\frac{q}{p}}} + f(x)\right) \right)}\]
where
\[ f(x) =  \left(\frac{1}{2q} - \frac{1}{2p} \right) \frac{\log{k(x)}}{k(x)} - q \frac{k(x)}{n(x)} = O\left(\frac{1}{\log_{2}{x}}\right). \]
Finally observe that
\[ k(x) = \frac{\log{x}}{\log_{2}{x}} \left( 1 + O\left( \frac{\log_{3}{x}}{\log_{2}{x}}\right) \right), \]
which completes the proof.
\end{proof}

\section{Application to multipliers}

Recall that a sequence of real numbers $(\lambda_{n})_{n \in \N}$ is said to be a \emph{multiplier} from $\mathcal{H}_{p}$ to $\mathcal{H}_{q}$, if for every Dirichlet series $\sum_{n}{a_{n}n^{-s}}$ in $\mathcal{H}_{p}$ we have that $\sum_{n}{\lambda_{n}a_{n}n^{-s}}$ belongs to $\mathcal{H}_{q}$. In \cite{Bayart1}, Bayart makes use of Weissler result \cite{Weissler} to obtain sufficient conditions for a multiplicative sequence $(\lambda_n)$ (i.e., $\lambda_{n m} = \lambda_{n}  \lambda_{m}$ for all $m,n$) to be a multiplier from $\mathcal{H}_{p}$ to $\mathcal{H}_{q}$. Here we use Theorem \ref{Theo:mainEquality} to give a sufficient condition for a not necessarily multiplicative sequence of positive real numbers to be a multplier.

\begin{Theo}Given $1 \leq p < q < \infty$,
let $(\lambda)_{n \in \N}$ be a decreasing sequence of positive real numbers satisfying
\[ \sum_{n}{\frac{\lambda_{n}}{n \log\log n} \left(\sqrt{\frac{q}{p}} + \varepsilon\right)^{\frac{\log{n}}{\log \log n}}} <  \infty  \hspace{3mm} \mbox{ for some $\varepsilon > 0$}. \]
Then $(\lambda_{n})_{n \ in  \N}$ is a multiplier from  $\mathcal{H}_{p}$ to $\mathcal{H}_{q}$.
\end{Theo}

\begin{proof}
Let us denote
\[
g(x):=\exp{\left( \frac{\log{x}}{\log_{2}{x}} A \right)} \, \mbox{ where } \, A :=\log{\sqrt{\frac{q}{p}}} + \varepsilon.
\]
Recall that there exists $C > 0$ such that for every $x>1$, the  partial sum operator $S_{x}(\sum_{n}{a_{n} n^{-s}}) = \sum_{n \leq x}{a_{n}n^{-s}}$ has norm $\| S_{x}\|_{\mathcal{H}_{p} \rightarrow \mathcal{H}_{p}} \leq C \log{x}$. Given $D = \sum_{n}{a_{n}n^{-s}}$ in $\mathcal{H}_{p}$, we then have
\[
 \left\| \sum_{n \leq x}{\frac{a_{n}}{n^{s}}} \right\|_{\mathcal{H}_q} \, \leq \mho(q,p,x) \, \left\| \sum_{n \leq x}{\frac{a_{n}}{n^{s}}} \right\|_{\mathcal{H}_p} \leq \mho(q,p,x) \, C \, \log{x}  \| D \|_{\mathcal{H}_{p}}.
 \]
By Theorem \ref{Theo:mainEquality}, we deduce that when $x$ is big enough
\begin{equation}
\label{equa:inequalityPartialSUms}
  \left\| \sum_{n \leq x}{\frac{a_{n}}{n^{s}}} \right\|_{\mathcal{H}_q} \, \leq g(x) \, \| D \|_{\mathcal{H}_{p}}.
\end{equation}
Moreover, also if $x$ tends to infinity we have that
\[ 0 \leq g'(x) \leq \frac{A \, g(x)}{x \log_{2}{x}} \: \mbox{ and so } \: \frac{d}{dx}{\left(\frac{ \, g(x)}{x \log_{2}{x}} \right)} \leq g(x) \left( A - \log_{2}{x} \right) < 0. \]
This means that for $n$ big enough,
\begin{equation}
\label{equa:boundDerivative}
g(n+1) - g(n) = \int_{n}^{n+1}{g'(x) \: dx}  \leq \frac{A \, g(n)}{n \, \log_{2}{n}}.
\end{equation}
Let $0 < m < M$ be natural numbers. Using Abel's summation formula
\[ \sum_{n=m}^{M}{\frac{\lambda_{n}a_{n}}{n^{s}}} = \sum_{n=m}^{M-1}{\left( \sum_{k=1}^{n}{\frac{a_{k}}{k^{s}}} \right) \left( \lambda_{n} - \lambda_{n+1} \right)} - \left( \sum_{k=1}^{m-1}{\frac{a_{k}}{k^{s}}} \right) \lambda_{m} + \left( \sum_{k=1}^{M}{\frac{a_{k}}{k^{s}}} \right) \lambda_{M}.\]
Therefore, taking $m$ big enough and using  \eqref{equa:inequalityPartialSUms} and \eqref{equa:boundDerivative}
\begin{align*}
\left\| \sum_{n=m}^{M}{\frac{\lambda_{n}a_{n}}{n^{s}}} \right\|_{q} & \leq \| D\|_{p} \left(  \sum_{n=m}^{M-1}{g(n) \left( \lambda_{n} - \lambda_{n+1} \right)} + g(m-1) \lambda_{m} + g(M) \lambda_{M}\right)\\
& \leq \|D\|_{p} \left( 2\lambda_{m}g(m) + \sum_{n=m}^{M-1}{\lambda_{n}(g(n+1) - g(n))}  \right)\\
& \leq \|D\|_{p} \left( 2 \lambda_{m} g(m) + A \sum_{n=m}^{M-1}{\frac{\lambda_{n}g(n)}{n \log_{2}{n}}} \right).
\end{align*}
The series $\sum_{n}{\frac{\lambda_{n}g(n)}{n \log_{2}{n}}}$ converges by hypothesis. On the other hand, it also follows from this fact that there is an increasing sequence $(N_{k})_{k \in \N}$ of natural numbers such that $\lim_{k}{\lambda_{N_{k}} g(N_{k})} = 0$. Hence, the inequality above leads to the existence of a subsequence of the partial sums of $\sum_{n}{\lambda_{n}a_{n} n^{-s}}$ converging in $\mathcal{H}_{q}$, which in particular means that $\sum_{n}{\lambda_{n}a_{n} n^{-s}} \in \mathcal{H}_{q}$.
\end{proof}

\section{Remarks}

One of the main tools in the proof of Theorem \ref{Theo:mainEquality} has been \eqref{equa:estimationExponent}. This estimation is also valid when we deal with $m$-homogenous polynomials with coefficients in an arbitrary (complex) Banach space (see \cite{PolynomialCotype}). This means that the argument in the proof of Theorem \ref{Theo:estimationAbove} also works for Dirichlet polynomials with coefficients in some complex Banach space.

Although probably without leading to a better estimate in Theorem \ref{Theo:mainEquality}, we strongly believe that the inequality from  \eqref{equa:estimationExponent} can be improved in the following way:
\begin{Conj}
For every $1 \leq p < q < \infty$ and $m \in \N$ we have that
\[ H^{q,p}_{m} \leq \, m^{\frac{1}{2q} - \frac{1}{2p}} \left(\sqrt{\frac{q}{p}}\right)^{m}.\]
\end{Conj}
Indeed, for the case in which $p<q$ are powers of two, we can use elementary methods to show that this conjecture is true. We sketch here the proof of the case $p=2$ and $q=4$:

Let $P = \sum_{|\alpha| =m}{c_{\alpha} \omega^{\alpha}}$ be an $m$-homogeneous polynomial. Given $\alpha, \gamma \in \N_{0}^{(\N)}$ we write $\alpha \leq \gamma$ whenever $\alpha _{n} \leq \gamma_{n}$ for each $n \in \N$. We this notation
\[ \| P\|_{2}^{4} = \left( \sum_{|\alpha| = m}{|c_{\alpha}|^{2}} \right)^{2} = \sum_{|\gamma|=m}{\left( \sum_{|\alpha|=m, \alpha \leq \gamma}{|c_{\alpha}|^{2}| |c_{\gamma - \alpha}|^{2}} \right)}. \]
\[
\|P\|_{4}^{4} = \sum_{|\gamma| = 2m}{\left| \sum_{|\alpha|=m}{c_{\alpha}c_{\gamma - \alpha}} \right|^{2}}
\leq \sum_{|\gamma|=2m}{\left( \sum_{\substack{|\alpha|=m \\ \alpha \leq \gamma}}{|c_{\alpha}|^{2} |c_{\gamma - \alpha}|^{2}} \right) \kappa(\gamma,m)}.
\]
where $\kappa(\gamma,m) = |\{ \alpha \colon |\alpha|=m, \alpha \leq \gamma \}|$. Among all $\gamma \in \N_{0}^{(\N)}$ with $|\gamma| = 2m$, we have that the maximum value of $\kappa(\gamma,m)$ is attained whenever the entries of $\gamma$ are all either one or zero. In this case, we can calculate explicitely $\kappa(\gamma,m)$ in terms of a combinatorial number that can be estimated by means of Lemma \ref{Lemm:EstimationCombinatorialNumber} as
\[ \kappa(\gamma,m) \leq \binom{2m}{m} \leq \frac{4^{m}}{\sqrt{\pi m}}. \]
We then conclude that
\[ \| P\|_{4}^{4} \leq \binom{2m}{m} \leq \frac{4^{m}}{\sqrt{\pi m}} \| P\|_{2}^{4} \]
which gives the desired result.

\end{document}